%% file: maisop_arxiv.tex
\documentclass[10pt]{article}
\usepackage[OT2,T1]{fontenc}
\usepackage[utf8]{inputenc}

\usepackage{hyperref}
\pdfoutput=1

\usepackage[nice]{nicefrac}
\usepackage{calc}
\usepackage{amsmath}
\usepackage{amsfonts}
\usepackage{amssymb}
\usepackage{amsthm}
\usepackage[usenames,dvipsnames]{color}
\usepackage{graphics}
\usepackage{graphicx}
\usepackage{pst-all}
\usepackage{subfig}	
\usepackage{amscd}
\usepackage{mathdots}

\usepackage{mathrsfs}
\usepackage{subfig}
\usepackage{array}

\usepackage{dsfont}

\input{myMath}

\newtheoremstyle{teorema}{8pt}{8pt}{}{}{\scshape}{}{8pt}{}
\theoremstyle{teorema}
\newtheorem{theorem}{Theorem}[section]
\newtheorem{proposition}[theorem]{Proposition}

\newtheorem{lemma}[theorem]{Lemma}

\newtheorem{problem}{Problem}
\newtheoremstyle{definizione}{8pt}{8pt}{\itshape}{}{\scshape}{}{8pt}{}
\theoremstyle{definizione}
\newtheorem{definition}[theorem]{Definition}

\DeclareMathOperator{\qin}{ \, \tilde \in\,  }

\def\Caption#1{\caption{\footnotesize{#1}}}

\begin{document}
\title{Optimal rank matrix algebras preconditioners}

 \author{F. Tudisco, C. Di Fiore\\
 \small{Department of Mathematics, University of Rome ``Tor Vergata'',}\\ \small{ Via della Ricerca Scientifica, 00133 Rome, Italy}\\[6pt] 
 E. E. Tyrtyshnikov\\ \small{Institute of Numerical Mathematics, Russian Academy of Sciences,}\\ \small{ Gubkina Street, 8, Moscow 119991, Russia}}
 
 \date{ }


\maketitle

\begin{abstract}
 When a linear system $A x = y$ is solved by means of iterative methods (mainly CG and
GMRES) and the convergence rate is slow, one may consider a preconditioner $P $ and move
to the preconditioned system $P^{-1}Ax = P^{-1}y$. The use of such preconditioner changes
the spectrum of the matrix defining the system and could result into a great
acceleration of the convergence rate. The construction of \textit{optimal rank} preconditioners is
strongly related to the possibility of splitting $A$ as $A = P + R + E$, where $E$ is a
small perturbation and $R$ is of low rank \cite{tyrty}. In the present work we extend the black-dot 
algorithm  for the computation of such splitting for $P$ circulant (see \cite{ot}), to the case where $P$ is in $\mss A$, for several known low-complexity matrix
algebras $\mss A$. The algorithm so obtained is particularly efficient when $A$ is Toeplitz plus Hankel like. 
We finally discuss in detail the existence and the properties of the decomposition  $A=P+R+E$ when $A$ is Toeplitz, also extending to the $\vphi$-circulant and Hartley-type cases some results previously known for $P$ circulant.
\end{abstract}

\section{Introduction}
In this work we consider a new approach to the construction of
preconditioning matrices $P$ for the solution of linear systems $Ax=y$.
We call this kind of preconditioners \textit{optimal rank} because they are
produced
trying to force the rank of $A-P$ to be as small as possible. Optimal rank circulant
preconditioners $P$ were initially proposed for Toeplitz systems by Tyrtyshnikov et al. in
\cite{ot}, \cite{zot}. Here we basically extend them to several known low 
complexity matrix algebras \cite{maiop}, \cite{bz}, \cite{maiop2}, \cite{dibenedetto} and to Toeplitz plus Hankel like matrices. 

\subsection{Notations}
We use some standard notations, which are briefly described here.

With $\msf M_n$ we denote the Hilbert space of $n\times n$ matrices whose entries are in the complex field $\CC$, with $\msf U_n$ the group of unitary $n\times n$ matrices and with $\msf M_n^+$ the cone of positive semi-definite $n\times n$ matrices. We use also the notation $A\geq 0$ for an element of $\msf M_n^+$ and $A>0$ for those matrices which are strictly positive definite.
Given $A \in \msf M_n$, $\sigma(A)$ and $\lambda_i(A)$ are the spectrum and the $i$-th eigenvalue of $A$, respectively.

The square bracket $[\,\cdot\,,\,\cdot\,]:\msf M_n\times \msf M_n\to \msf M_n$ denotes the commutator 
$$A,B\mapsto [A,B]=AB-BA$$
also denoted by $\nabla_A(B)=[B,A]$. The round bracket $(\,\cdot\,, \,\cdot\,)$  denotes the standard scalar product on $\CC^n$ (sometimes it can denote the scalar product on different Hilbert spaces; it will be clear from the context).

 The symbol $e_i$ is used for the $i$-th canonical vector $(e_i)_k = 1$ if $k=i$ and $(e_i)_k=0$ otherwise.

Finally given a matrix $W \in \msf M_n$ we shall use the symbol $\mss A(W)$ for the algebra generated by $W$, namely the closed set 
$$\mss A(W)={\{p(W)\mid p \text{ polynomials} \}}$$
\section{Low complexity matrix algebras}\label{sec:algebras}
Suppose we are given a unitary matrix $U \in \msf U_n$. Then it can be naturally defined the algebra of normal matrices
$$\mss A = \sd U = \{U\diag(\theta_1,\dots, \theta_n)U^*\mid \theta_i \in \CC\}\, $$
also called \textit{algebra of matrices simultaneously diagonalized by a unitary
transform} or briefly $\sd U$ algebra. For an element $A \in \mss A$, the complexity of the
products $A\times vector$ and $A^{-1}\times vector$ depends only on the complexity of
$U\times vector$ and $U^*\times vector$. Therefore we say that a $\sd U$ algebra $\mss A$
is of low complexity if both these products are computable with less than $O(n^2)$
operations, in particular if they can be performed with $O(n \log n)$ floats.


Notice that, whenever $W \in \msf M_n$ is diagonalized by $M \in \msf M_n$, for any $A \in
\mss A(W)$ we have $A = M\diag(\lambda_1(A), \dots, \lambda_n(A))M^{-1}$. Moreover if
$(A_\lambda)$ is a family of mutually commuting matrices, it is known that
$(A_\lambda)$ admits a common Schur basis. These two facts together imply that the
algebra $\mss A(N)$ generated by a normal $n\times n$ matrix $N$, must satisfy $\mss A(N)
\subset \sd U$
for a unitary matrix $U \in \msf U_n$.  The inclusion can be proper, and precisely
it is an identity if and only if  $N$ is \textit{non-derogatory}\footnote{A matrix $A\in
\msf M_n$ is said to
be non-derogatory if $\deg(p)\geq n$, for any polynomial $p$ such that $p(A)=0$, or, equivalently, if the geometric multiplicity of any eigenvalue of $A$ is one. We make often use of matrices that are both normal and non-derogatory, which therefore are those
normal
matrices that have pairwise different eigenvalues.}
\cite{maiop}. Therefore the
non-derogatorycity hypothesis on the matrix $W$ leads to the following further
characterization
$$\mss A(W) = \ker \nabla_W $$
and if $W$ is normal we also have $\mss A(W)=\sd U$, for an $U \in \msf U_n$.

The best known low complexity $\sd U$ algebras are commonly divided
into three classes: $\vphi$-circulants, Trigonometric and Hartley-type. 
Some specific choices in such classes have been used successfully to solve linear algebra problems involving Toeplitz matrices or, more generally, structured matrices related to shift invariance of the  mathematical
model considered (see f.i. \cite{chan-cg} and references
therein). Anyway, depending on the problem, any algebra among the three families could find a potential application as a preconditioner.

\subsection{$\vphi$-circulant algebras}\label{sec:circ}
Let us consider the matrix 
\begin{equation}\label{Pi-vphi}
 \mPi_\vphi = \matrix{cccc}{ & 1 &  &  \\ & & \ddots & \\ & & & 1 \\ \vphi &  &  & },
\quad \vphi \in \CC\, ,
\end{equation}
which is the basis for the definition of the family of $\vphi$-circulant algebras
\begin{definition}
 Given $\vphi \in \CC$, the algebra $\C_\vphi = \mss A(\mPi_\vphi)$ generated by $\mPi_\vphi$ is called $\vphi$-circulant algebra.
\end{definition}
Note that $e_1^\t \mPi_\vphi^k = e_{k+1}^\t$ for $k=0,1,\dots,n-1$, which implies that $\C_\vphi$ is a $1$-space \cite{maiop} or, in other words, that any matrix $C \in \C_\vphi$ is uniquely defined by its first row. It is now natural to introduce the operator 
$$C_\vphi: \CC^n \To \C_\vphi, \qquad x \mapsto C_\vphi(x)$$
which maps $x \in \CC^n$ into the matrix $C_\vphi(x) \in \C_\vphi$ whose first row is~$x^\t$. 

It is not difficult to observe that the matrix
\begin{equation}\label{F-vphi}
 F_\vphi =\frac{1}{\sqrt n}\Bigl( \vphi^{\f i n}\omega^{ij} \Bigr)_{i,j=0,1,\dots,n-1}, \quad \omega=e^{-\f{2\pi \i}n}
\end{equation}
diagonalizes the algebra $\C_\vphi$, namely that
$$\C_\vphi = \{F_\vphi \diag(\theta_1,\dots,\theta_n)F_\vphi^{-1}\mid \theta_i \in \CC \}\, .$$
Moreover $F_\vphi \in \msf U_n$ if and only if $|\vphi|=1$ (cf.  \cite{davis_circulant}). From now on we assume that $\vphi$, defining $\C_\vphi$, has modulus one, unless otherwise specified. Nevertheless we underline that several formulas that we obtain could be adapted to the case of a generic complex $\vphi$.  The choice $\vphi=1$ gives rise to the well known circulant algebra $\C$, diagonalized by the Fourier matrix 
\begin{equation}\label{Four}
 \textstyle{F =\frac{1}{\sqrt n}\Bigl( e^{-\f{2\pi kh \i}n} \Bigr)_{k,h=0,1,\dots,n-1}}
\end{equation}
which is naturally related to $F_\vphi$ up to the diagonal scaling
$$F_\vphi = \diag(1,\vphi^{\f 1 n}, \dots, \vphi^{\f{n-1} n})F = \mDelta_\vphi F\, .$$
We shall use simply the symbol $\mPi$ for $\mPi_1$. Another very popular choice is $\vphi=-1$ which defines the so called skew-circulant algebra $\C_{-1}$. Both $\C$ and $\C_{-1}$  and their applications has been widely studied, see for instance \cite{huckle}. Since $\mPi_\vphi$ is normal and non-derogatory we have $\C_\vphi = \ker \nabla_{\mPi_\vphi}\, .$

\subsection{Trigonometric algebras} 
There are sixteen different trigonometric algebras presently known \cite{dibenedetto}, \cite{bozzo-difiore}.
Eight of them are diagonalized by a discrete sine-type transform, the other eight by a
cosine-type one. We start again by introducing a family of matrices
\begin{equation}\label{trig}
 X_\mu = X_{(\mu_1,\mu_2,\mu_3,\mu_4)} = \matrix{ccccc}{
  \mu_1 & \mu_2  &        &        &   \\
   1    & 0      & 1      &        &   \\
        & \ddots & \ddots & \ddots &   \\
        &        & 1      &      0 & 1 \\
        &        &        & \mu_3  & \mu_4 } , \quad \mu=\matrix{c}{\mu_1 \\ \mu_2 \\ \mu_3\\ \mu_4}\, ,
\end{equation}
and defining $\T_\mu = \mss A(X_\mu)$. Sixteen different choices for the vector $\mu \in \RR^4$ give rise to the sixteen trigonometric algebras (see f.i. appendix 1 in \cite{scarabotti}). We list in Table \ref{algebre_trig} such particular values for $\mu$ naming the corresponding trigonometric algebras $\T_\mu$ also  $DST$ or $DCT$ so to recall that they are diagonalized by a discrete sine transform or a discrete cosine transform, respectively. 
\begin{table}[!h]
\centering
\Caption{ Sixteen choices for  $\mu\in \RR^4$ and respective trigonometric algebras.}\label{algebre_trig}
\begin{tabular}{c|c|c|c|c}
 & $\mu_3=2$ & $\mu_3=1$ & $\mu_3=1$& $\mu_3=1$\\
 & $\mu_4=0$ &  $\mu_4=0$ &  $\mu_4=1$ &  $\mu_4=-1$\\
\hline
$\mu_1=0$, $\mu_2=2$ & $DCT_1$ & $DCT_3$ & $DCT_5$ & $DCT_7$\\
\hline
$\mu_1=0$, $\mu_2=1$ & $DST_3$ & $DST_1$ & $DST_7$ & $DST_5$\\
\hline
$\mu_1=1$, $\mu_2=1$ & $DCT_6$ & $DCT_8$ & $DCT_2$ & $DCT_4$\\
\hline
$\mu_1=-1$, $\mu_2=1$ & $DST_8$ & $DST_6$ & $DST_4$ & $DST_2$
\end{tabular}
\end{table}
It is not difficult to observe that for such choices of $\mu$ the matrix $X_\mu$ is normal
and non-derogatory, since $\mu_2\neq 0$, so $\T_\mu = \ker\nabla_{X_\mu}$.

The well known \textit{tau-algebra} is $DST_1 = \T_{(0,1,1,0)}=\T$ whose generating matrix will be denoted simply by $X$. Such algebra was considered in \cite{bini-capovani} where it is defined as the set of $n\times n$ matrices satisfying the \textit{cross-sum} rule with null boundary conditions
\begin{align*}
 \me T =\{A \in \msf M_n\mid &a_{ij}\in \CC, \,a_{i-1,j}+a_{i+1,j}=a_{i,j-1}+a_{i,j+1}, \, i,j=1,\dots, n \\
			     &a_{n+1,j}=a_{0,j}=a_{i,n+1}=a_{i,0}=0 \}.
\end{align*}
This is a computationally useful definition and it clearly is nothing but the scalar form of
our previous characterization of $\T$ as the kernel of $\nabla_X$. Other trigonometric
algebras which attained particular attention are $DCT_2$, $DST_2$, $DST_{3}$, $DST_7$ mainly because of their applications to 
 image processing \cite{imageappl} and displacement decomposition \cite{bozzo-difiore}, \cite{dfz-displacement}. 

As for the $\vphi$-circulants, also the trigonometric algebras are $1$-spaces
\cite{maiop}, namely each matrix in $\T_\mu$ is uniquely defined up to its first row. In 
\cite{dfz-displacement} such property is explicitly shown for the more generic set of
Hessenberg algebras which contains both $\vphi$-circulant and trigonometric algebras. As a
special case of that result we derive the following characterization for $\T_\mu$:
consider $n-1$ polynomials $\Phi_1, \dots,\Phi_{n-1}$, each $\Phi_k$
defined as the characteristic polynomial of the  principal submatrix of $X_\mu$ of order
$k$. Setting $\Phi_0=1$ and
$$\mc X_\mu^{(k)} = \mu_2^{-1}\Phi_{k-1}(X_\mu), \qquad k=1,\dots,n\, ,$$
then $\T_\mu = \Span(\mc X_\mu^{(1)},\dots,\mc X_\mu^{(n)})$ and the $1$-space property
$e_1^\t \mc X_\mu^{(k)}=e_k^\t$ follows. So, as for the circulant case, we introduce in a
natural way the operator
$$\tau_\mu : \CC^n \to \T_\mu \qquad x \mapsto \tau_\mu(x)$$
which maps $x \in \CC^n$ into the matrix $\tau_\mu(x)\in \T_\mu$ whose first row is
$x^\t$.

Since $X_\mu$ is non-derogatory and normal as well, there exists a matrix $U_\mu \in \msf U_n$ such that $\T_\mu = \sd U_\mu$. Using the symbol $S_\mu$ ($C_\mu$) for the matrix $U_\mu$ diagonalizing $\T_\mu$ when the choice of $\mu$ gives rise to a $DST$ ($DCT$), we list $S_\mu$, $C_\mu$ and the eigenvalues of $X_\mu$ in Tables \ref{tabella_seni}, \ref{tabella_coseni}. We finally underline that all of them satisfy the low complexity property (see \cite{dct}, \cite{dst} and references therein).
\begin{table}[!b]
\setlength{\extrarowheight}{3pt}
\centering
\Caption{Discrete sine transform $S_\mu$ and the eigenvalues $\lambda_k$ of $X_\mu$, $k,h=0,\dots,n-1$ }\label{tabella_seni}
\begin{tabular}{c|c|c}
       & $(S_\mu)_{kh}$ & $\lambda_k(X_\mu)$ \\
\hline
$DST_1$ & $\displaystyle{\sin(k+1)(h+1)\f \pi {n+1}}$ & $\displaystyle{\cos\frac{(k+1)\pi}{n+1}}$\\
\hline
$DST_2$ & $\displaystyle{\sin(k+1)\(h+\um\)\f \pi {n}}$ & $\displaystyle{\cos\frac{(k+1)\pi}{n}}$\\
\hline
$DST_3$ & $\displaystyle{\sin\(k+\um\)(h+1)\f \pi {n}}$ & $\displaystyle{\cos\frac{(k+\um)\pi}{n}}$\\
\hline
$DST_4$ & $\displaystyle{\sin\(k+\um\)\(h+\um\)\f \pi {n}}$ & $\displaystyle{\cos\frac{(k+\um)\pi}{n}}$\\
\hline
$DST_5$ & $\displaystyle{\sin(k+1)(h+1)\f \pi {n+\um}}$ & $\displaystyle{\cos\frac{(k+1)\pi}{n+\um}}$\\
\hline
$DST_6$ & $\displaystyle{\sin(k+1)\(h+\um\)\f \pi {n+\um}}$ & $\displaystyle{\cos\frac{(k+1)\pi}{n+\um}}$\\
\hline
$DST_7$ & $\displaystyle{\sin\(k+\um\)(h+1)\f \pi {n+\um}}$ & $\displaystyle{\cos\frac{(k+\um)\pi}{n+\um}}$\\
\hline
$DST_8$ & $\displaystyle{\sin\(k+\um\)\(h+\um\right)\f \pi {n-\um}}$ & $\displaystyle{\cos\frac{(k+\um)\pi}{n-\um}}$
\end{tabular}
\end{table}
\begin{table}[!t]
\setlength{\extrarowheight}{3pt}
\centering
\Caption{Discrete cosine transform $C_\mu$ and the eigenvalues $\lambda_k$ of $X_\mu$, $k,h=0,\dots,n-1$}\label{tabella_coseni}
\begin{tabular}{c|c|c}
       & $(C_\mu)_{kh}$ & $\lambda_k(X_\mu)$ \\
\hline
$DCT_1$ & $\displaystyle{\cos k h \f \pi {n-1}}$ & $\displaystyle{\cos\frac{k\pi}{n-1}}$\\
\hline
$DCT_2$ & $\displaystyle{\cos k \(h+\um\) \f \pi {n}}$ & $\displaystyle{\cos\frac{k\pi}{n}}$\\
\hline
$DCT_3$ & $\displaystyle{\cos \(k+\um\) h \f \pi {n}}$ & $\displaystyle{\cos\frac{(k+\um)\pi}{n}}$\\
\hline
$DCT_4$ & $\displaystyle{\cos \(k+\um\) \(h+\um\) \f \pi {n}}$ & $\displaystyle{\cos\frac{(k+\um)\pi}{n}}$\\
\hline
$DCT_5$ & $\displaystyle{\cos k h \f \pi {n-\um}}$ & $\displaystyle{\cos\frac{k\pi}{n-\um}}$\\
\hline
$DCT_6$ & $\displaystyle{\cos k \(h+\um\) \f \pi {n-\um}}$ & $\displaystyle{\cos\frac{k\pi}{n-\um}}$\\
\hline
$DCT_7$ & $\displaystyle{\cos \(k+\um\) h \f \pi {n-\um}}$ & $\displaystyle{\cos\frac{(k+\um)\pi}{n-\um}}$\\
\hline
$DCT_8$ & $\displaystyle{\cos \(k+\um\) \(h+\um\right) \f \pi {n+\um}}$ & $\displaystyle{\cos\frac{(k+\um)\pi}{n+\um}}$
\end{tabular}
\end{table} 
\subsection{Hartley-type algebras}
In \cite{bz} eight different unitary matrices $H_i$ are introduced and eight Hartley-type
algebras $\H_i = \sd H_i$ are defined as the set of matrices simultaneously diagonalized
by such $H_i$. The Hartley algebra $\H_1$ was introduced in \cite{bini-favati}
and the  well known Hartley matrix $H_1$, which diagonalizes it, is defined as follows
$$H_1  = \frac{1}{\sqrt n}\[\cos\(\f{2\pi i j} n \)+ \sin\(\f {2\pi i j} n\)\]_{ij=0,1,\dots,n-1}\, .$$
The multiplication $H_1\times vector $ can be performed with $O(n \log n)$ operations, nonetheless in \cite{bz} is shown that the same low complexity property holds for all the $H_i$, $i=1,\dots,8$. Let us introduce in detail the Hartley-type transforms and the corresponding algebras. Consider the matrices
\begin{gather*}
 K = H_2 = \frac{1}{\sqrt n}\[\cas\(\f{2\pi i (2j+1)} n \right)\right]_{ij,=0,1,\dots,n-1} \\
 G = H_3 = \frac{1}{\sqrt n}\[\cas\(\f{2\pi (2i+1) (2j+1)} {2n} \right)\]_{ij,=0,1,\dots,n-1}\, ,
\end{gather*}
where $\cas x = \cos x + \sin x$. Both of them are orthonormal matrices. Furthermore consider the two sparse $n\times n$ matrices
$$E_1 = \f 1 {\sqrt 2}\matrix{cccc}
{\sqrt 2 & & & \\
 & I &  & J \\
 &   & \sqrt 2 & \\
 & -J & & I }, \qquad 
E_2 = \f 1 {\sqrt 2}\matrix{ccc}
{I &  & -J \\
  & \sqrt 2 & \\
  J & & I }
$$
where the presence of the central row and column depends on the oddness of $n$. Consider
finally the four subspaces of $\pm 1$-circulant algebras
$$\C_{\pm 1}^s =\{C \in \C_{\pm 1} \mid C^\t = C \}, \quad \C_{\pm 1}^{sk} = \{C \in \C_{\pm 1} \mid C^\t = -C\}\, ,$$
the first two being a proper subalgebra of $\C_{\pm 1}$. 

The eight Hartley-type transforms and algebras are defined by the identities in Table \ref{tab:hartley}.
Note that only the algebra $\ma G = \sd G$ is not a $1$-space, in fact its elements are not defined uniquely by the first row (for more details see \cite{maiop}, \cite{bz}).
\begin{table}[h!]
\setlength{\extrarowheight}{2pt}
\Caption{Definitions for the eight Hartley-type algebras $\H_i$}\label{tab:hartley}
\centering
\begin{tabular}{l}
	$\H_1 = \me H = \sd H  = \sd H_1 = \me C^s + J \mPi \me C^{sk}$\\
\hline
	$\H_2 = \me K = \sd K = \sd H_2= \me C_{-1}^s + J \mPi_{-1} \me C_{-1}^{sk}$\\
\hline
	$\H_3 = \ma G = \sd G = \sd H_3= \me C_{-1}^s + J \me C_{-1}^{sk}$\\
\hline
	$\H_4 = \tilde {\me K} = \sd K^\t  = \sd H_4 = \me C^s + J \me C^{sk}$\\
\hline
	$\H_5 = \eta = \sd (K^\t E_1) = \sd H_5= \me C^s + J \me C^{s}$\\
\hline
	$\H_6 = \mu = \sd (GE_2) = \sd H_6= \me C_{-1}^s + J  \me C_{-1}^{s}$\\
\hline
	$\H_7 = \alpha = \sd (HE_1^\t ) = \sd H_7= \me C^s + J \mPi \me C^{s}$\\
\hline
	$\H_8 = \beta = \sd (KE_2^\t ) = \sd H_8= \me C_{-1}^s + J \mPi_{-1} \me C_{-1}^{s}$
\end{tabular}
\end{table}

Unlike $\vphi$-circulant and trigonometric algebras, it is not so clear that  Hartley-type algebras can be
introduced as the algebras generated by matrices whose structure is predictable for all $n$.  However, since all of them are algebras of normal matrices simultaneously diagonalized
by a unitary transform, there exist non-derogatory matrices $W_i$ such that $\H_i =
\mss A(W_i) = \ker\nabla_{W_i}$. Nevertheless let us note that the $\C^s_{\pm 1}$
part of $\H_i$ is the subalgebra $\mss A(Y_{\pm 1})$ generated by the following
derogatory matrix 
\begin{equation}\label{Y_phi}
 Y_{\vphi} = \mPi_{\vphi}+\mPi_{\vphi}^\t = \matrix{cccc}{0 & 1 & & \vphi\\
 1&   & \ddots & \\
  &  \ddots &        & 1 \\
 \vphi & & 1 & 0}\, .
\end{equation} Noting that 
$$F_\vphi^* \mPi_{\vphi} F_\vphi = \mOmega_\vphi = \vphi^{1/n}\mOmega = \vphi^{1/n}\diag(\omega^i\mid
i=0,\dots,n-1), \qquad \omega=e^{-\f{2\pi \i}n}$$
and that $\mPi_{\pm 1} + \mPi^\t_{\pm 1} = \mPi_{\pm 1} + \mPi^*_{\pm 1} = 2\Re \mPi_{\pm 1}$,\footnote{The real part (Hermitian part) of a matrix $X$ is the Hermitian matrix $\Re X=\um (X+X^*)$} we easily obtain an explicit
formula for the eigenvalues of $Y_1$ and  $Y_{-1}$, namely
\begin{equation}\label{eigY}
 \lambda_k(Y_1)=2 \cos \(\f{2k\pi} n\), \quad \lambda_k(Y_{-1}) = 2 \cos \(\f{(2k+1)\pi}
n\)
\end{equation}
for $ k=0,\dots,n-1$.
\section{Optimal rank preconditioning}
Let us consider an $n\times n$ linear system
\begin{equation}\label{system}
 A x = y, \qquad A \in \msf M_n,\quad  x,y \in \CC^n
\end{equation}
which should be solved by some Krylov subspace iterative methods (CG and GMRES are good
examples). When the convergence rate of such methods is low, one may consider a
suitable preconditioning matrix $P$, switch to the preconditioned linear system
\begin{equation}\label{prec_system}
 P^{-1}A x =P^{-1}y\, ,
\end{equation}
and apply to this new system the iterative methods. Clearly, except for trivial cases, the spectrum of the matrix defining the system \eqref{prec_system} is different from the original one, while the solution $x$ maintains unchanged. The introduction of the  preconditioner $P$ could lead to a substantial improvement of the convergence rate, provided that $P$ satisfies some \textit{``good''} properties, that we summarize into the following two
\begin{enumerate}
 \item The condition number $\kappa(P^{-1}A)$ is uniformly bounded in $n$
\item The spectrum of $P^{-1}A$ has a cluster around $1$.
\end{enumerate}
When $A> 0$ (the general case needs some further hypothesis cf. \cite{tyrty-zam-yu}), the first property leads to the linear convergence of the methods, independently on the dimension $n$ of the problem. The second one is related with the super-linear convergence of the methods and, fixed an $\ep >0$, it is the same as requiring that the following splitting  for the matrix $A$ in \eqref{system} 
$$A = P + R +E$$
holds, being $E$ a small perturbation, $\|E\|\leq \ep$, and $R$ a low-rank matrix, namely $\rank
R$ is $o(n)$  \cite{tyrty}, \cite{ot}. It is
clear that the cluster of $\sigma(P^{-1}A)$ is a proper cluster whenever $\rank R=r_\ep(n)$ is
uniformly bounded with respect $n$, in fact $r_\ep(n)$ is exactly the number of
eigenvalues of $P^{-1}A$ which are outside a ball of radius $\ep$ around $1$. Moreover we
can heuristically affirm that \textit{``the smaller is $\rank R$, the smaller is the
cluster of $\sigma(P^{-1}A)$ and the better is the preconditioner $P$''}.

In the following we consider low complexity algebras of matrices simultaneously
diagonalized by a unitary fast transform $U$, or rather closed sets of the form
$$\mss A = \sd U = \{U\diag(\theta_1,\dots, \theta_n)U^*\mid \theta_i \in \CC\}\, $$
where $U \times vector$ and $U^*\times vector$ require $O(n\log n)$ operations.

Our considerations about  property $2$ above suggest the definition of an optimal rank
matrix algebra preconditioner for a given linear system \eqref{system}.
\begin{definition}
 Given an invertible $n\times n$  matrix $A$ and a matrix algebra $\mss A$, we call optimal rank preconditioner in $\mss A$ for $A$ any matrix 
$$\A =\arg \min \{ \mathrm{rank}(A-P-E)\mid P \in \mss A, \|E\|\leq \ep \}\, .$$
\end{definition}
In \cite{ot} it is stressed that in order to construct such preconditioner for $A$
 one should solve the following
\begin{problem}\label{pr_1}
 Given $A \in \msf M_n$, given an $\sd U$ algebra $\mss A$ and given $\ep>0$, find $\A = P
+ R$, with $P \in \mss A$ and $\|A-\A\|\leq \ep$ such that $\rank R$ is as
small as possible.
\end{problem}
Observe that for any given $M \in \sd U$ there exists a diagonal matrix $D$ such that $M = U D U^*$. Therefore if the norm considered in Problem~\ref{pr_1}  $\|\, \cdot \, \|$ is unitarily invariant 
$$\|A-\A\| = \|U^*AU - D - \tilde R\|, \qquad \rank \tilde R = \rank R\, .$$
In other words we can split problem \ref{pr_1} into the following problems \ref{pr_2} and
\ref{pr_3} and calculate the minimum over the algebra of diagonal $n\times n$
matrices $\mss D$ 
\begin{problem}\label{pr_2}
 Given $A \in \msf M_n$ and $\ep>0$, find $\hat \A = D + R$, with $D \in \mss D$ and
$\3 A-\hat \A\3 \leq \ep$ such that  $\rank R$ is minimum, for a chosen unitarily
invariant norm $\3 \, \cdot\, \3$.
\end{problem}
\begin{problem}\label{pr_3}
 Given $A \in \msf M_n$ and given an $\sd U$ algebra $\mss A = \sd U$, compute the image
$U^*AU$.
\end{problem}
If $\mss A = \sd U$ is of low complexity, the computation of $U^*AU$ requires an amount of $O(n^2\log n)$ operations which is not acceptable. For this reason we have posed  Problem \ref{pr_3}. However, we shall see that in solving Problem \ref{pr_3} it is not necessary to compute all the entries of the image matrix $U^*AU$, it is instead enough to have an algorithm that computes any prescribed entry $(U^*AU)_{ij}$ in a fast way, by a number of operations independent of the matrix size.

Problem \ref{pr_2} can be approached as in the circulant-Toeplitz case \cite{ot}, i.e. by means of the \textit{black-dot
algorithm}, an  ad-hoc version of the incomplete cross algorithm \cite{skeleton},
\cite{skeleton3}, \cite{boundary}. Given an algebra $\mss A =\sd U =\mss A(W)$ ($W$
non-derogatory), in order to apply the black-dot algorithm we need to know some elements
$(U^*AU)_{ij}$, where the pair $(i,j)$ belongs to a certain set of indices $\Omega$.
Typical choices for $\Omega$ give rise to a method whose complexity can be estimated with
$O(nr_\ep(n)^2)$ \cite{ot}. When the coefficients matrix is positive definite we expect
that the preconditioner $\A$ computed by such algorithm is positive definite as well.
Actually such property is not always ensured, in fact it may happen that some entry of the computed diagonal matrix $D$ is negative. However, in typical cases - see \cite{ot} and the proposition here below - it is possible to ensure
the positive definitess of $\A$ by applying a low-rank correction to the computed $D$, i.e. modifying a small number of its diagonal entries.
\begin{proposition}
 Let $A=P+R+E$ with $\rank R=r<n$ and $\|E\|<\ep$. If $A> \ep I$ and  $R\geq 0$ then at
least $n-r$ eigenvalues of $P$ are positive.\\
\end{proposition}

\begin{proof} Weyl's inequalities for the Hermitian eigenvalues problem $Z=X+Y$ give us the following
\begin{gather*}
\lambda_{i+j-1}(Z)\leq \lambda_i(X)+\lambda_j(Y)\, ,\qquad 1\leq i+j-1\leq n\, ,\\
\lambda_i(X)+\lambda_n(Y)\leq \lambda_i(Z)\leq \lambda_i(X)+\lambda_1(Y)\, ,\qquad 1\leq i\leq n \, ,
\end{gather*}
where the eigenvalues of a Hermitian matrix are supposed in decreasing order (i.e. $\lambda_i\geq \lambda_{i+1}$). Since $r<n$, we have $\lambda_{r+k}(R)=0$ for $k=1,\dots, n-r$. Thus Weyl's inequalities applied to $A=P+R+E$ give us
$$\lambda_{i+r}(A)=\max_{ k = r,\dots, n-1}\lambda_{i+k}(A)\leq \lambda_i(P+E)\leq \lambda_i(A), \qquad 1 \leq i\leq n-k\, .$$ 
Therefore we know that all the first $n-r$ eigenvalues of $P+E$ are greater or equal to $\ep$, thus $n-r$ eigenvalues of $P$ are positive, since $A>\ep I$.
\end{proof}
  In the next section we discuss Problem \ref{pr_3} and propose a method for computing
any element $(U^*AU)_{ij}$ at a very low cost, after a preprocessing phase
of complexity $O(n\log n)$. Such method works well if the displacement rank of $A$,
the matrix defining our system \eqref{system}, with respect to the chosen algebra $\mss A$, is
sufficiently small. More precisely let us introduce the following
\begin{definition}
 Let $\mss A = \mss A(W)$ be a matrix algebra. We say that a matrix $A \in \msf M_n$ \textit{almost-belongs} to $\mss A$, in symbols $A \qin \mss A$, if $\rank([A, W])$ is uniformly bounded in $n$. 
\end{definition}
Note that by the fundamental theorem of homomorphism we have a canonical isomorphism $\dot
\nabla_W$ between $\nicefrac{\msf M_n}{\mss A}$ and $\range \nabla_W$ which implies that
the pre-image of $A \in \range \nabla_W$ is given by the closed set $\{\dot
\nabla_W^{-1}(A)+\A\}_{\A \in \mss A}$.


The method proposed in the next section for computing cheaply the elements $(U^*AU)_{ij}$, works if $A \qin \mss A$ and  $[A,W]$ is explicitly known. We underline since now that our method works well for all the low complexity $\sd U$ algebras $\mss A$ previously
presented and for any  Toeplitz or Hankel matrix~$A$. 

\section{The computation of $U^*AU$}
Let $\mss A\subset \msf M_n$ be an algebra of normal matrices $\mss A = \sd U$, $U \in \mb
U(n)$. Let us consider an element $W\in \mss A$ which is non-derogatory. Clearly $\mss A =
\mss A(W)=\ker \nabla_W$ and $U^*WU\in \mss D$. Set $D =U^*WU$. Given a matrix $A \in \msf M_n$, we have
$$(U^*AU)D - D(U^*AU) = U^*[A,W]U$$
thus the off-diagonal elements of $U^*AU$ satisfy the identity
\begin{equation}\label{image}
 (U^*AU)_{ij} = \frac{(U^*[A,W]U)_{ij}}{\lambda_{j}(W)-\lambda_{i}(W)}, \quad i\neq
j\, .
\end{equation}
 Such equalities let us state the theorem below, whose detailed proof can be derived also by the observations which follow.
\begin{theorem}\label{calcolo_immagine}
 Let $\mss A=\mss A(W)$ be a low complexity $\sd U$ algebra, $A \qin \mss A$ and
$\rank [A,W]=\rho$ (thus $\rho$ is uniformly bounded in
$n$ and $\lambda_i(W)=\lambda_k(W)$ $\sse$ $i=k$). Assume that $\sigma(W)$ is explicitly
known. Then, after a preprocessing phase of complexity $O(n\log n)$ required for the
computation of $U^*[A,W]U$ (see below), each off-diagonal element $(U^*AU)_{ij}$ can be calculated with
$\rho +1$ multiplications. 
\end{theorem}
The theorem above gives us a tool for approaching Problem \ref{pr_3}, or rather for 
computing  the image of a matrix $A$ under a unitary fast transformation. In fact, we
are requested to compute $2\rho$ transformations $\hat x_k= U^* x_k$, $\hat y_k =U^*y_k$, for $k=1,\dots,\rho$,
in a preprocessing phase, where $x_k$ and $y_k$ are the vectors defining a dyadic decomposition of $[A,W]$. Then each off-diagonal entry of $U^*AU$ is known up to $\rho+1$ multiplicative operations  
$(U^*AU)_{ij}=(\lambda_j(W)-\lambda_i(W))^{-1}\sum_{k=1}^\rho (\hat x_k)_i\bar{(\hat y_k)_j}$. Finally observe that in
order to apply the black-dot algorithm we actually do not need the diagonal entries of
$U^*AU$ \cite{ot}, thus Theorem \ref{calcolo_immagine} is not restrictive in our
situation.

To summarize, in order to compute  the optimal rank preconditioner $\A$ for
$A$ into the algebra $\mss A = \sd U = \ker \nabla_W$,  we can propose the following
DR-scheme, where the black-dot algorithm is like a black box which  sometime requires an entry $(U^*AU)_{ij}$, $i\neq j$.\\
\linea
\textsf{DR-scheme}\\
\linea
Assume $\sd U = \ker \nabla_W$ and $[A,W]=\sum_{k=1}^{\rho}x_k y_k^*$.
\begin{enumerate}
\item (Preprocessing) Compute the $2\rho$ fast transforms $\hat x_k = U^*x_k$, $\hat
y_k = U^*y_k$, $k=1,\dots,\rho$.
\item Start the black-dot iterations
\begin{itemize}
 \item[2.1.] \textbf{if} $(U^*AU)_{ij}$ \textbf{is required}, compute it via the
identity
$$(U^*AU)_{ij} = \frac{\sum_{k=1}^\rho (\hat x_k)_i \bar{(\hat
y_k)_j}}{\lambda_j(W)-\lambda_i(W)}$$
where $\bar z$ denotes the complex conjugate of $z$.
\end{itemize}
\item Proceed with the iterations until convergence, passing through 2.1 if necessary
\end{enumerate}
\linea
Note that the preprocessing phase requires $O(\rho n \log n)$ operations, whereas only
$\rho +1$ arithmetic multiplications are needed each time step 2.1 must be performed.

Despite the somewhat general formulation of the DR-scheme, in the
following we consider some specific cases, in which $A$ is a Toeplitz or Hankel matrix,
and discuss how to compute $U^*AU$ explicitly, for the low complexity algebras presented
in Section \ref{sec:algebras}, underlining that for these choices of $A$, we actually have
$A \qin \mss A$.

\subsection{$\vphi$-circulant algebras}
Recall that the generic $\vphi$-circulant algebra is defined as 
$$\C_\vphi = \mss A(\mPi_\vphi) = \sd F_\vphi = \ker \nabla_{\mPi_\vphi}\, ,$$
where $F_\vphi = \mDelta_\vphi F = \diag(1, \vphi^{\f 1 n}, \dots, \vphi^{\f{n-1}n})F$, $|\vphi|=1$, and $F$ is the Fourier matrix \eqref{Four}.
We will make use of the following matrix $J$, also called \textit{reverse identity}
$$J = \matrix{cccc}{ & & & 1\\ & & 1 & \\ & \iddots & & \\ 1 & & &}\, .$$
Given $p \in \ZZ$, consider the equivalence class $\{p \mod n\}$ and let $[p]_n$ denote its unique representative in $\{0,1,\dots,n-1\}$.
\begin{proposition}\label{rank_phicirc}
 Let  $T_n = (t_{i-j})_{ij}$ be a Toeplitz matrix and $\mPi_\vphi$ the matrix \eqref{Pi-vphi} generating $\C_\vphi$. Then $T_n \qin \C_\vphi$ and                                                                                                                 
 \begin{equation}
   [T_n,\mPi_\vphi]=x_\vphi \, e_1^\t + e_n\,y_\vphi^\t \label{toepl-circ}    
 \end{equation}
being $x_\vphi = (\vphi t_{1-n}- t_1,\dots,\vphi t_{-1}- t_{n-1},0)^\t$ and  $y_\vphi = -Jx_\vphi$.     
\end{proposition}
\begin{proof}
Due to the definition of $\mPi_\vphi$ we have the following equality
$$[T_n, \mPi_\vphi]_{ij} = \sum_{k=0}^{n-1}\(t_{i-k}(\mPi_\vphi)_{kj}-(\mPi_\vphi)_{ik}t_{k-j}\)=\vphi t_{i-[j+n-1]_n}- t_{[i+1]_n- j}\, .$$
Therefore \eqref{toepl-circ} holds for $i \neq n-1$ and $j\neq 0$. In fact for such choices of indexes we have $i-[j+n-1]_n = i-(j-1)=[i+1]_n -j$ thus $[T_n , \mPi_\vphi]_{ij}=0$ and $\rank([T_n, \mPi_\vphi])=2$. It is not difficult to observe that \eqref{toepl-circ} also holds for $i=n-1$ and $j=0$.
\end{proof}
The above Proposition  immediately implies that, when $A$ has Toeplitz structure, the DR-scheme can be applied to the
case $U=F_\vphi$, the unitary matrix diagonalizing $\C_\vphi$. The Hankel case needs some further observation since  the rank of $[H_n,\mPi_\vphi]$ is not bounded in general and a direct use of the DR-scheme would be prohibitive. Given a Hankel matrix $H_n$ call $T(H_n)$ the Toeplitz matrix $JH_n$. Here below we observe that when $\vphi$ is $1$ or $-1$, the computation of $F^*_{\vphi} H_nF_{\vphi}$ can be brought back to the Toeplitz case $F^*_{\vphi}T(H_n) F_{\vphi}$, for which, instead, the DR-scheme works well. 

Observe that 
$$\mDelta_\vphi^* J \mDelta_\vphi = \diag\(\vphi^{\f {n-2k-1}n}\mid k=0,\dots,n-1\right)J\, ,$$
therefore when $\vphi \in \{-1,1\}$ we get the equality $F^*_{\vphi} JF_{\vphi} = \vphi^{\f{n-1}n}F^*JF$. Moreover, since $F = J\mPi F^*$, $F^2 = J\mPi$ and $F^*JF=\mOmega$, we also have 
$$F^*JF = (F^*\mPi F)(F^*)^2 = \mOmega J \mPi\, ,$$
where $\mOmega=\diag(1,\omega, \dots, \omega^{n-1})$, $\omega = e^{-2\pi \i /n}$. Now use the definition of $T(H_n)$ to write
$$F^*_\vphi H_n F_\vphi  = (F^*_\vphi JF_\vphi )(F^*_\vphi T(H_n) F_\vphi)\, .$$ 
The formulas obtained so far imply the desired result:
\begin{gather*}
 (F^*_{\vphi}H_nF_{\vphi})_{ij}=\vphi^{\f{n-1}n}\omega^{[-i]_n}(F^*_{\vphi}T(H_n) F_{\vphi})_{[-i]_n,j}\, ,\\
 i,j=0,\dots,n-1,\,\, i\neq j\, .
\end{gather*}
%

Before  proceeding to our discussion for the trigonometric and Hartley cases, let us introduce some further useful notation. Given $a,b \in \CC^n$ with $a_1=b_1$, we shall denote with $T_n(a,b)$ the Toeplitz matrix whose first column is $a$ and whose first row is $b^\t$. Analogously, given $u,v \in \CC^n$ such that $u_n = v_1$ we shall denote with $H_n(u,v)$ the Hankel matrix whose first row is $u^\t$ and whose last column is $v$. By Proposition \ref{rank_phicirc} the following formula holds for any $a,b \in \CC^n$ with $a_1=b_1$
\begin{equation}\label{TPi}
  [T_n(a,b),\mPi_\vphi]= \GE_\vphi(a,b) - J  \GE_\vphi(a,b)^\t J, 
\end{equation}
where $\GE_\vphi(a,b)$ is the rank one matrix 
\begin{equation}\label{GE}
 \GE_\vphi(a,b) = (\vphi Jb - \mPi_\vphi a)e_1^\t\, .
\end{equation}

\subsection{Trigonometric algebras}
As discussed in Section \ref{sec:algebras} all the sixteen trigonometric algebras $\T_\mu$ are generated by the matrix $X_\mu$ \eqref{trig}, for  the sixteen choices of $\mu \in \RR^4$ shown in Table \ref{algebre_trig}. Being $X_\mu$ non-derogatory and normal all its eigenvalues  are distinct and we know them explicitly (Tables \ref{tabella_seni} and \ref{tabella_coseni}). As a consequence, Theorem \ref{calcolo_immagine} holds for the trigonometric algebras and for the set of matrices $A \in \msf M_n$ such that $A \qin \mss \T_\mu$. Let us show that all the Toeplitz and Hankel matrices belong to such set.

Given $\mu \in \RR^4$, let us split $X_\mu$ into $X_\mu = X+M_\mu$, where
$$X = \matrix{ccccc}{
   & 1  &        &        &   \\
   1    &      & 1      &        &   \\
        & \ddots &  & \ddots &   \\
        &        & 1      &       & 1 \\
        &        &        & 1 &  },\quad  M_\mu= \matrix{ccccc}{
  \mu_1 & \mu_2-1  &        &        &   \\
      &      &      &        &   \\
        &        &     &       &  \\
        &        &        & \mu_3-1  & \mu_4 }\, .$$ 
For any two vectors $a,b \in \CC^n$ with $a_1=b_1$ we obviously have $[T_n(a,b), X_\mu]=[T_n(a,b),X]+[T_n(a,b), M_\mu]$. Thus  it is possible to prove the result for the tau algebra $\T$ and then for all the other algebras $\T_\mu$.
\begin{proposition}\label{rank_tau}
  Let $A \in \msf M_n$ be  Toeplitz, Hankel or the sum of them. Then $A \qin \T_\mu$ for any  $\mu \in \CC^4$, precisely $\rank [A,X_\mu]\leq 8$.
\end{proposition}
\begin{proof}
  Notice that  $\rank(M_\mu)\leq 2$ for all $\mu \in \CC^4$, thus $\rank([A,M_\mu])\leq 4$. Moreover note that given  any Toeplitz  $T=(t_{i-j})_{ij}$ or Hankel $H=(h_{i+j})_{ij}$, the  matrix $T+H$ satisfies the cross-sum rule with non-null boundary conditions, that is
$$t_{i-1-j}+h_{i-1+j}+t_{i+1-j}+h_{i+1+j}-(t_{i-j+1}+h_{i+j-1}+t_{i-j-1}+h_{i+j+1})=0$$
for any $i,j \in \ZZ$. Therefore the boundary conditions are given exactly by the border columns and rows of the $(n+2)\times (n+2)$ matrix embedding the given $T+H$ $n\times n$ matrix and maintaining its same structure. Such columns and rows  can not be null except for trivial cases. As a consequence, since $A \in \T$ if and only if it satisfies the cross-sum with null boundary conditions, we have that  $\rank[T_n+H_n, X]\leq 4$, namely 
$$[T_n+H_n,X]=\matrix{c|c|c}{\ast & \cdots & \ast\\
\hline
\vdots & O_{n-2} & \vdots\\
\hline
\ast & \cdots & \ast}$$
where $\ast$ are in general non-null entries  and $O_{n-2}$ is the null matrix of order $n-2$.
\end{proof}
\begin{proposition}
 Let $a,b, c, d \in \CC^n$ with $a_1=b_1$ and $c_n=d_1$. Then
\begin{gather}
 [T_n(a,b),X]=\mTheta(a,b)-J\mTheta(a,b)^\t J\, ,\label{TX} \\
 [H_n(c,d), X] = \mTheta(d, Jc )J-J\mTheta(d,Jc)^\t \label{HX}
\end{gather}
where $\mTheta(a,b)$ is the rank two matrix
\begin{equation}\label{theta}
 \mTheta(a,b)= e_1 (\mPi b)^\t - (\mPi a)e_1^\t\, .
\end{equation}
\end{proposition}
\begin{proof}
 Displacement formula \eqref{HX} for $H_n(c,d)$ clearly follows from the
previous one \eqref{TX}, since  $H_n(c,d)=T_n(d,Jc)J$ and 
$$[H_n(c,d), X]=[T_n(d,Jc),X]J=\mTheta(d, Jc)J-J\mTheta(d,Jc)^\t \, .$$
By Proposition \ref{rank_tau} only the border rows and columns of $[T_n,X]$ are non-null, for a Toeplitz matrix $T_n$. Thus we just need to check \eqref{TX} for such four vectors. For instance, noting that $ X = \mPi_0 + \mPi_0^\t$, we have
$$[T_n(a,b), X]e_1 = T_n(a,b)e_2 - Xa =  \mPi_0^\t a + b_2e_1- Xa =b_2 e_1- \mPi_0 a $$
and analogously 
$$e_1^\t[T_n(a,b), X] = b^\t X - e_2^\t T_n(a,b) = ( \mPi_0 b)^\t - a_2e_1^\t\, .$$ 
Therefore \eqref{TX} holds for the first row and column, if we define $\mTheta(a,b)$ with $\mPi_0$ in place of $\mPi$. It is not difficult to observe that the same can be said also for the the last row and column. Moreover, thanks to arithmetic cancellations for the corner positions $(1,n)$ and $(n,1)$, identity \eqref{TX} holds for our definition of $\mTheta(a,b)$, given in terms of $\mPi$. 
\end{proof}
The above proposition  gives us an explicit formula for the displacement rank of a Toeplitz
or Hankel matrix into the tau-algebra $\T$. Then, if $S = S_{(0,1,1,0)}$ is the sine
transform which diagonalizes $\T$, we can apply the DR-scheme to the case where $U=S$ and
$A=T_n+H_n$ is the sum of any two Toeplitz and Hankel matrices.

Anyway, by Proposition \ref{rank_tau}, the displacement rank of a Toeplitz or Hankel matrix into any trigonometric algebra $\T_\mu$ does not exceed $8$. Let us first derive an explicit formula for the dyadic decomposition of $[T_n, X_\mu]$. Since 
$$M_\mu = e_1(\mu_1e_1+(\mu_2-1)e_2)^\t + e_n ((\mu_3-1)e_{n-1}+ \mu_4 e_n)^\t\, ,$$
we have 
\begin{align*}
 [T_n(a,b), M_\mu] &= c_1(a,b)\Bigl( \mu_1 e_1 + (\mu_2-1)e_2 \Bigr)^\t \\
&+ c_n(a,b)\Bigl( (\mu_3-1)e_{n-1} + \mu_4 e_n \Bigr)^\t \\
		   & - e_1\Bigl(\mu_1 c_1(a,b) + (\mu_2-1)c_2(a,b)\Bigr)^\t \\
&- e_n \Bigl( (\mu_3-1)c_{n-1}(a,b)+ \mu_4 c_n(a,b)\Bigr)^\t
\end{align*}
where $c_k(a,b)$ is the $k$-th column of $T_n(a,b)$. The required formula is obtained by summing the latter one and \eqref{TX}. A similar computation provides a formula for $[H_n,X_\mu]$ and thus the DR-scheme can be applied when $U=S_\mu, C_\mu$ is any trigonometric transform and $A=T_n+H_n$.
\subsection{Hartley-type algebras}

Let $\H_k = \sd H_k$ denote a generic Hartley-type algebra, we can characterize $\H_k$ as
the
set \cite{df-displacement}
\begin{equation}\label{H}
 \H_k = \{A \mid [A,Y_{\vphi}]=[A,M_k]=0\}
\end{equation}
where $\vphi \in \{1,-1\}$ and $M_k \in \msf M_n$ depend on the Hartley-type algebra. We
are considering. Despite it is obviously possible to define
$\H_k$ as the set $\ker \nabla_W$ for some non derogatory matrix $W$, it is not always
easy
to find a ``simple'' $W$ with such property. Therefore here we make use of \eqref{H} and
derive an easy variant of the DR-scheme. 

\begin{theorem}
 Let $X$ and $Y$ be two distinct normal matrices which commute. Let $\lambda_i(X)$, $\lambda_i(Y)$ be the eigenvalues of $X$, $Y$ corresponding to the same common eigenvector. Consider the algebra $\mss A = \{A \mid
[A,X]=[A,Y]=0\}$. Then $\mss A$ is $n$ dimensional if and only if
$\lambda_i(X)=\lambda_j(X)$ implies $\lambda_i(Y)\neq \lambda_j(Y)$, for any pair of
distinct indices $(i,j)$.
\end{theorem}
\begin{proof}
Denote with $\lambda_1, \dots,\lambda_n$ and $\mu_1, \dots,\mu_n$ the eigenvalues of $X$
and $Y$, respectively. Since $[X,Y]=0$ we have $X = UD_\lambda U^*$ and $Y = U D_\mu U^*$,
where $D_\lambda$ and $D_\mu$ are the diagonal matrices such
that $(D_\lambda)_{ii}=\lambda_i$ and $(D_\mu)_{ii}=\mu_i$.

First let us prove the  implication ($\Leftarrow$). Take $A \in \mss A$, then
$U^*A U D_\lambda = D_\lambda U^*AU$ and $U^*AUD_\mu = D_\mu U^*AU$. Let
$(U^*AU)_{ij}=\hat a_{ij}$. By writing the first relation entrywise we get $\hat
a_{ij}(\lambda_j - \lambda_i)=0$, therefore $\hat a_{ij}$ must be zero if $\lambda_i \neq
\lambda_j$. When $\lambda_i = \lambda_j$ we use the second relation obtaining $\hat
a_{ij}(\mu_j - \mu_i)=0$, which gives us $\hat a_{ij}=0$ due to our hypothesis. As a
consequence we have $\hat a_{ij}=0$ $\forall i \neq j$, hence $U^*AU \in \mss D$ and $\dim
\mss A = \dim \sd U = n$. 

Viceversa, since any $A \in \sd U$ is an element of $\mss A$ and
$\dim \mss A = n$, we have $\mss A = \sd U$. Now proceed by absurd and assume the claim to be
false. Without loss of generality, suppose $\lambda_1 = \lambda_2$ and $\mu_1 =
\mu_2$, and consider the matrix
$$A = U \matrix{cccc}{d_1 & b & & \\
			  & d_2 & & \\
 & & \ddots & \\
 & & & d_n}U^* = UBU^*$$
where $b \neq 0$. We have $U^*[A,X]U = BD_\lambda - D_\lambda B=0$ and $U^*[A,Y]U =
BD_\mu - D_\mu B=0$. Therefore $A$ commutes with both $X$ and $Y$, namely $A \in \mss
A$. This is impossible since $A \notin \sd U$.
\end{proof}
By the above theorem it is clear how to adapt  the DR-scheme for Hartley-type algebras\\ 
\linea
Assume $[A,Y_{\vphi}]=\sum_{s=1}^{\rho}x_s y_s^*$ and $[A, M_k]=\sum_{s=1}^\tau w_s
z_s^*$.
\begin{enumerate}
\item (Preprocessing) Compute the $2(\rho+\tau)$ fast transforms $\hat x_s = H_k^* x_s$,
$\hat
y_s = H_k^* y_s$, $\hat w_t = H_k^* w_t$, $\hat z_t = H_k^* z_t$, $s=1,\dots,\rho$, $t=1,\dots,\tau$.
\item Start the black-dot iterations
\begin{itemize}
 \item[2.1.] \textbf{if} $(H_k^*AH_k)_{ij}$ \textbf{is required}, compute it via the
identity
$$(H_k^*AH_k)_{ij} = \frac{\sum_{s=1}^\rho (\hat x_s)_i \bar{(\hat
y_s)_j}}{\lambda_j(Y_{\vphi})-\lambda_i(Y_{\vphi})}$$
if $\lambda_i(Y_{\vphi})\neq \lambda_j(Y_{\vphi})$, or  
$$(H_k^*AH_k)_{ij} = \frac{\sum_{s=1}^\tau (\hat w_s)_i \bar{(\hat
z_s)_j}}{\lambda_j(M_k)-\lambda_i(M_k)}$$
otherwise
\end{itemize}
\item Proceed with the iterations until convergence, passing through 2.1 if necessary
\end{enumerate}
\linea
Note that in this modified version the DR-scheme requires one more assumption, since the
matrix $A$ we are considering should have small displacement rank with respect to two
different matrices $Y_\vphi$ and $M_k$, rather than only one. Of course there are
many matrices which satisfy such assumption, however it is not obvious if among them there are also Toeplitz or Hankel matrices. The rest of the
section is devoted to observe that this is true in particular cases.
 
First of all, note that the matrices $Y_\vphi$ are derogatory and only $\dim \C_{\vphi}^s$
of their
eigenvalues are distinct, in fact $Y_\vphi$ is the generator of the family of algebras
$$\C_\vphi^s = \{C \in \C_\vphi\mid C=C^\t\}\, .$$
Note that in \eqref{eigY} we derived an explicit formula for the spectrum of $Y_{\pm 1}$.

Moreover it is clear from Propositions \ref{rank_phicirc} that for any
Toeplitz matrix $A$ we have $\rank[A,Y_\vphi]\leq 4$. The same conclusion holds for $[A,Y_\vphi]$ with $A$ Hankel. Namely
\begin{proposition}
 Assume $|\vphi|=1$. Let $a,b \in \CC^n$ such that $a_1 = b_1$. Then 
$$[T_n(a,b), Y_\vphi]= \GE_\vphi(a,b) - \GE_\vphi(b,a)^\t + J\Bigl(\GE_\vphi (b,a) - \GE_\vphi(a,b)^\t\Bigr)J$$
where $\GE_\vphi(x,y)$ is the rank one matrix in \eqref{GE}.

Let $c,d \in \CC^n$ with $c_n = d_1$. Then 
$$[H_n(c,d), Y_\vphi] = e_1 d^\t (J\mPi - \vphi I) - (J\mPi -\vphi I)d e_1^\t + e_n c^\t (\mPi J - \vphi I)  - (\mPi J -\vphi I)c e_n^\t$$
\end{proposition}
\begin{proof}
 The formula for $T_n(a,b)$ immediately follows from \eqref{TPi}, noting that $T_n(a,b)^\t = T_n(b,a)$ and that  $[A,B^\t]=-[A^\t, B]^\t$, $\forall$ $A,B \in \msf M_n$.
The second one, for $H_n(c,d)$, follows from \eqref{HX} and the decomposition $Y_\vphi = X + R_\vphi$, being  $R_\vphi = \vphi(e_ne_1^\t + e_1e_n^\t)$. In fact $H_n(c,d)e_n = d$, $H_n(c,d)e_1 = c$, $e_1^\t H_n(c,d)=c^\t$ and $e_n^\t H_n(c,d)=d^\t$. 
 \end{proof}
Clearly we can also represent $[T_n(a,b), Y_\vphi]$ by means of $\mTheta(a,b)$, in fact
$$[T_n(a,b), Y_\vphi] = \mTheta(a,b)-J\mTheta(a,b)^\t J + \vphi J(be_1^\t - e_1b^\t) +\vphi (a e_1^\t - e_1 a^\t)J$$
which easily comes from \eqref{TX}.

Now, concerning $M_k$, let us fix four choices among the eight possible indices $k$, 
precisely $k = 1,2,5,6$, and consider the corresponding Hartley-type algebras $\H =
\H_1$, $\K = \H_2$, $\eta=\H_5$ and $\mu=\H_6$.
For such choices we can characterize $M_k$ somehow explicitly, in fact in \cite{maiop}
it is shown that 
\begin{equation}\label{Mk}
 M_k = J + \matrix{cc}{0 & 0^\t \\ 0 & \tau(z_k)}
\end{equation}
where  $\tau(z_k) = \tau_{(0,1,1,0)}(z_k) \in \mb
M(n-1)$ and the vectors $z_k \in \RR^{n-1}$ are, respectively,
$$\textstyle{z_1 = \um (e_2-e_{n-1}), \quad z_2 = -\um (e_2+e_{n-1}), \quad z_5 = z_6 =
0}\, .$$
As a consequence we get $\tau(z_5)=\tau(z_6)=O_{n-1}$, and
$$\textstyle{\tau(z_1)=\um X(I-J)\, ,\qquad \tau(z_2)=-\um X(I+J)}$$
where $X=X_{(0,1,1,0)}\in \mb M(n-1)$ is defined in \eqref{trig}. Also note that for the
algebras $\H_5 = \eta$ and $\H_6 = \mu$ a more elegant
characterization does hold. We state it by means of the following 
\begin{proposition}
 Let $\eta$ and $\mu$ be the Hartley-type algebras defined in Table \ref{tab:hartley}.
Then
$$\eta = \ker \nabla_{Y_1 + J}\qquad \mu = \ker \nabla_{Y_{-1}+J}$$
\end{proposition}
\begin{proof}
Clearly $\eta \subset \ker \nabla_{Y_1 + J}$,  $\mu \subset \ker \nabla_{Y_{-1}+J}$ and
the equalities hold if and only if $Y_1 + J$ and $Y_{-1} + J$ are non-derogatory. If $H_5$
is the unitary matrix diagonalizing $\eta$, it is not difficult to observe that
$$H_5^*JH_5 = \matrix{cc}{I_m & \\ & -I_{n-m}}$$
where $m=n/2$ if $n$ is even and $m=(n+1)/2$ otherwise. This remark and \eqref{eigY} imply that the eigenvalues $\lambda_i(Y_1+J)=\lambda_i(Y_1)+\lambda_i(J)$ are all
distinct; thus $Y_1 + J$ is non derogatory, and the thesis follows for $\eta$. In the
same way one proves the thesis also for $\mu$.
\end{proof}
It is now clear that we can apply the modified DR-scheme to a quite general class of
Toeplitz and Hankel matrices. In fact for
any symmetric Toeplitz matrix $T_n$ and any persymmetric Hankel matrix $H_n$
we have $[T_n,J]=[H_n,J]=0$, therefore $[T_n,XJ]=[T_n,X]J$ and $[H_n,XJ]=[H_n,X]J$. Using
\eqref{TX} and \eqref{HX}, it is now straightforward to derive the formulas for the
commutator of $T_n$ and $H_n$ with $M_k$, for $k=1,2,5,6$, taking into account that the
matrix $X$ which appears into \eqref{Mk} has order $n-1$. This eventually allows us to
apply the modified DR-scheme to the case $A = $ Toeplitz symmetric $+$ Hankel
persymmetric.

We conclude this section by noting  that using only $Y_{\vphi}$ and formula \eqref{image}
we can compute at least $O(n(n-2))$ entries of $H_i^*AH_i$, for any $i =1,\dots,8$. The elements $\bullet$ of $H_i^*AH_i$ that we can not compute this way, are in the positions 
shown in the figure below (we represent them for $n=5,6$). 
{\footnotesize\begin{align}
&\matrix{cccccc}{
\bullet &  &  &  &  & \\
 & \bullet &  &  &  & \bullet\\
 &  & \bullet &  & \bullet & \\
 &  &  & \bullet &  & \\
 &  & \bullet &  & \bullet & \\
 & \bullet &  &  &  & \bullet
}_{ even} \qquad
\matrix{ccccc}{
\bullet &  &  &  & \\
 & \bullet &  &  & \bullet\\
 &  & \bullet & \bullet & \\
 &  & \bullet & \bullet & \\
 & \bullet &  &  & \bullet
}_{ odd}\tag{$\vphi=1$}\\
\footnotesize &\matrix{cccccc}{
\bullet &  &  &  &  & \bullet\\
 & \bullet &  &  & \bullet & \\
 &  & \bullet & \bullet &  & \\
 &  & \bullet & \bullet &  & \\
 &  \bullet&  &  & \bullet & \\
\bullet &  &  &  &  & \bullet
}_{  even} \qquad
\matrix{ccccc}{
\bullet &  &  &  & \bullet \\
 & \bullet &  & \bullet & \\
 &  & \bullet &  & \\
 & \bullet &  & \bullet & \\
\bullet &  &  &  & \bullet
}_{  odd}\tag{$\vphi=-1$}
\end{align}}
%
%

\section{Algebra-plus-low-rank approximation of a matrix}
As we underlined the previous section, an optimal rank preconditioner $P$ should realize the
splitting 
\begin{equation}\label{s}
 A = P + R + E
\end{equation}
where $\|E\|\leq \ep$ and $\rank R$ is minimum. When $A$ is
Toeplitz, it can be shown that optimal or Strang-type preconditioners $P$, chosen inside
suitable $\sd U$ algebras, realize  an analogous decomposition where in general $\rank R =
O(\ep^{-p})O(n)$, for a $p>0$ \cite{tyrty}, \cite{maiop}, \cite{chan-cg}. In this section we show that for particular classes of matrices $A$ and algebras $\mss A$ there exists $P\in \mss A$ that realizes the splitting \eqref{s} with $\|E\|\leq \ep $ and $\rank R = o(n)$.

If  $\TT$ is the unit circle $\TT = \{z \in \CC \mid |z|=1\}$, let us denote with $T_n, H_n:L^\infty(\TT, \CC)\to \msf M_n$ the Toeplitz and Hankel operators, respectively, which map $f\in L^\infty(\TT, \CC)$ into the $n\times n$ Toeplitz or Hankel matrices $T_n(f)$, $H_n(f)$.  Finally call $\L(\TT)$ the subset of $L^\infty(\TT,\CC)$ of all piecewise holomorphic functions with logarithmic singularities, i.e. functions given by an holomorphic function plus a function with logarithmic singularities. A generic $f\in \L(\TT)$ has the form 
 $$f(z) = g(z)+ \sum_{k=0}^p \sum_{h=0}^q \alpha_{kh}\cdot (z - z_h)^k \log(z-z_h), \quad z \in \TT\, ,$$
with $g$ holomorphic over a set containing $\TT$ and $z_h \in \TT$, $h=0,\dots,q$.

 In \cite{zot} it is shown that for an $f \in \L(\TT)$ the Toeplitz matrix $T_n(f)$ admits the decomposition 
$$T_n(f) = P + R + E, \qquad \rank(R) = O\(\log\f 1 \ep\(\log \f 1 \ep + \log n\)\)$$  
where $P$ is a circulant matrix and  $E$ is a small perturbation as usual. The space $\L(\TT)$ is a special class of symbol functions that, however, covers all examples considered in literature on superlinear preconditioners \cite{ot}.

In this section we will show that a splitting analogous to the one in \cite{zot} also holds for $H_n(f)$ and $P$ chosen inside a generic $\vphi$-circulant  algebra, $\vphi \in \TT$. We will explicitly describe such matrix $P$ in several cases, and  we will discuss also the case of Hartley-type algebras.

Let us consider a $\lambda \in \CC$ and define the vector
$$p(\lambda)=\matrix{ccccc}{1 & \lambda & \lambda^2 & \cdots & \lambda^{n-1}}^\t\, . $$
It is not difficult to observe that the Toeplitz matrix
$$T_n(p(\lambda), \lambda^{1-n}Jp(\lambda))= p(\lambda)p(\lambda^{-1})^\t$$
 is a rank one matrix. Moreover, by requiring a Toeplitz matrix to be of rank one, we observe in fact that 
$$\rank T_n(a,b) = 1 \iff \exists \lambda \, \mid \, a = p(\lambda) \,\,  \text{and} \,\,  b=\lambda^{1-n}Jp(\lambda)\, .$$

Given a real number $\lambda$ set 
$$Z_n(\lambda)=T_n(p(\lambda),e_1)=\matrix{cccc}{
1	&	&	&	\\
\lambda	& 1	&	&	\\
\vdots	& \ddots	&\ddots	&	\\
\lambda^{n-1} &	\cdots &\lambda	& 1}.$$ 
This is the Toeplitz matrix generated by the symbol
$$\zeta_\lambda(\theta)=\f 1 {1-\lambda e^{\i \theta}},\quad \lambda,\theta\in \RR\, .$$
By noting  that  $p(\lambda)$ satisfies 
\begin{equation*}
 \mPi_\vphi p(\lambda) = \lambda p(\lambda) +(\vphi-\lambda^n)e_n
\end{equation*}
we obtain the following identities 
\begin{align*}
 \GE_\vphi(p(\lambda), e_1)  &= \vphi e_ne_1^\t -\mPi_\vphi p(\lambda)e_1^\t = (\lambda^n e_n - \lambda p(\lambda))e_1^\t\, , \\
\GE_\vphi(p(\lambda), \lambda^{1-n}Jp(\lambda))&=\Bigl( (\vphi\lambda^{1-n}-\lambda)p(\lambda) - (\vphi - \lambda^n) e_n\Bigr)e_1^\t
\end{align*}
which lead to
\begin{lemma}\label{C+R}
For any $\lambda \in \RR$ there exists $P_\vphi \in \C_\vphi$ such that the triangular Toeplitz matrix $Z_n(\lambda)$ splits into $Z_n(\lambda) = P_\vphi + R$, with~$\rank R~=~1$.
\end{lemma}
\begin{proof}
 Using \eqref{TPi} and the formulas above we have 
$$[T_n(p(\lambda), e_1), \mPi_\vphi]=\Bigl[\f{\lambda^n}{\lambda^n - \vphi}T_n(p(\lambda), \lambda^{1-n}Jp(\lambda)), \mPi_\vphi\Bigr]\, .$$
Notice, in fact, that, after the subtraction between $\GE_\vphi$ and  $J\GE_\vphi^\t J$, the term $ \vphi e_ne_1^\t$ is canceled.
\end{proof}
Observe that we can explicitly write the $\vphi$-circulant matrix $P_\vphi$ in Lemma \ref{C+R}, in fact it is given by the difference
\begin{equation}\label{C-vphi}
 P_\vphi = T_n(p(\lambda), e_1) - \f{\lambda^n}{\lambda^n - \vphi}T_n(p(\lambda), \lambda^{1-n}Jp(\lambda))\, .
\end{equation}
Hence, using the operator $C_\vphi:\CC^n \to \C_\vphi$ defined in Section \ref{sec:circ}, we have the equality
$$Z_n(\lambda) = C_\vphi(x_\vphi(\lambda)) + R, \qquad x_\vphi(\lambda) = \f 1
{\vphi-\lambda^n }J\mPi_\vphi p(\lambda) \, .$$
As a consequence, the following characterization holds
\begin{proposition}
 The eigenvalues of $C_\vphi(x_\vphi(\lambda))$ are
$$\zeta_{\f \lambda {\vphi^{1/n}}}\(\f {2\pi k}n\), \quad k=0,\dots, n-1\,.$$
\end{proposition}
\begin{proof}
 Observe that for a generic $\nu$ the following formulas hold
\begin{equation*}
 \mPi_\vphi p(\nu) = \nu p(\nu)+(\vphi -\nu^n)e_n, \qquad J p(\nu) = \nu^{n-1}p(\nu^{-1})\, .
\end{equation*}
Then note that
\begin{gather*}
 \sqrt n F_\vphi^\t x_\vphi(\lambda) = \f {\sqrt n} {\vphi-\lambda^n} F \mDelta_\vphi
J\mPi_\vphi p(\lambda) \\
=\f {\sqrt n} {\vphi-\lambda^n}F \mDelta_\vphi (\lambda J
p(\lambda) + (\vphi -\lambda^n)e_1) = \f {\lambda^n \sqrt
n}{\vphi-\lambda^n}F\mDelta_\vphi p(\lambda^{-1}) + e\\
=\f {\lambda^n} {\vphi-\lambda^n} \matrix{c}{\vdots \\ \f
{1-\vphi\lambda^{-n}}{1-\lambda^{-1}\vphi^{1/n}\omega^{k}}\\ \vdots}+ e=\matrix{c}{\vdots
\\ \f 1 {1- \lambda \vphi^{-1/n}\omega^{-k}}\\ \vdots}\, .
\end{gather*}
The thesis follows by recalling that, given $y \in \CC^n$, the eigenvalues of $C_\vphi(y)$ are the entries of the vector $\sqrt n F_\vphi^\t y$. 
\end{proof}
Consider now the matrix
$$K_n(\lambda) = Z_n(\lambda)+Z_n(\lambda)^\t -I, $$
i.e. the well known Kac-Murdock-Szego (KMS) matrix generated by the symbol
\begin{equation}\label{k}
 \kappa_\lambda(\theta) = 2\Re \zeta_\lambda(\theta)-1= \sum_{n\in \ZZ}\lambda^{|n|}e^{\i n \theta} = \f{1-\lambda^2}{1-2\lambda \cos\theta +\lambda^2} \, .
\end{equation}

Observe that when $\lambda$ is real, we have the identity, $K_n(\lambda) = 2 \Re Z_n(\lambda)-I$. It is also easy to show that $\kappa_\lambda(\theta)>0\iff |\lambda|<1$ or, equivalently, $K_n(\lambda)>0 \iff |\lambda|<1$. Therefore
\begin{proposition}\label{KMS}
 For a given $\lambda \in \RR$, let $K_n(\lambda)$ be the corresponding KMS matrix. Then
\begin{enumerate}
 \item For any $\vphi \in \TT$ there exist $Q_\vphi \in \C_\vphi$ and $R$ of rank 2 such that $K_n(\lambda)=Q_\vphi + R$
\item Let $\xi_\vphi(\lambda)=\( \f 1 {\vphi-\lambda^n}J\mPi_\vphi + \f {\bar \vphi}{\bar \vphi - \lambda^n}\right)p(\lambda) - 1$. Then $Q_\vphi =C_\vphi(\xi_\vphi(\lambda))$
\item The eigenvalues of $Q_\vphi$ are $\kappa_{\f{\lambda}{\vphi^{1/n}}}\(\f{2\pi k}n \right)$, $k=0,\dots,n-1$
\end{enumerate}
Notice that we also have $K_n(\lambda)>0\iff Q_\vphi>0$.
\end{proposition}
\begin{proof}
 By Proposition \ref{C+R} we have the equality $K_n(\lambda)= (2\Re C_\vphi(x_\vphi(\lambda)) -I)+R$
where $R$ is a rank two Hermitian matrix. Note that the matrix $\Re C_\vphi(x_\vphi(\lambda))$ belongs to $\C_\vphi$, and this proves (1). Moreover its eigenvalues are the real part of the eigenvalues
of $C_\vphi(x_\vphi(\lambda))$. Therefore from \eqref{k} we derive (3). Concerning (2)
just observe that, for any $y\in \CC^n$, $C_\vphi(y)^*=\bar \vphi C_\vphi(J\mPi_\vphi \bar
y)$, 
and use such remark to compute the first row of $2\Re C_\vphi(x_\vphi(\lambda)) - I$.
\end{proof}
It is important to note that the $\vphi$-circulant matrix $Q_\vphi$ in the previous
proposition is indeed the optimal rank $\vphi$-circulant preconditioner for a KMS
matrix. This fact can be easily proved by a direct calculation. Just try to impose
that the difference between $K_n(\lambda)$ and a rank one matrix is  $\vphi$-circulant to
reach an absurd. Therefore we have an explicit formula for the optimal rank
preconditioner of $K_n(\lambda)$. Note that it outperforms, from the clustering point of view, any
other known preconditioner for a KMS matrix \cite{chan-cg}, \cite{strang2}. In fact the preconditioned matrix $Q_{\vphi}^{-1}K_n(\lambda)$ has only three distinct eigenvalues.

\begin{theorem}\label{razio}
Let $p, q$ be two complex valued mutually prime polynomials defined on $\TT$, such that  $0 \notin q(\TT)$, $\deg p < \deg q$, and $q$ has all distinct roots. Then, for all $\vphi \in \TT$, the lower triangular Toeplitz matrix generated by $p/q$ satisfies the identity
$$T_n(p/q) = P_\vphi + R, \quad P_\vphi \in \C_\vphi, \quad \rank R\leq \deg p +1\, \, .$$ 
Moreover, if the roots of $q$ are real, then also the Hermitian Toeplitz matrix generated by $\Re(p/q)$ splits into
$$T_n(\Re (p/q)) = Q_\vphi + \tilde R, \quad Q_\vphi \in \C_\vphi, \quad \rank \tilde R\leq 2\rank R\, \, .$$
\end{theorem}
\begin{proof}
By the fundamental theorem of algebra, any polynomial $f:\TT\to \CC$ admits the splitting $f(z) = \prod_{i=1}^{\deg f}(z-z_i)$. As a consequence,  the rational function $\f {p(z)}{q(z)}$  admits the simple fractions decomposition
$$\ts{\f {p(z)} {q (z)}= \sum_{i=1}^{\deg q} \f {\rho_i}{z-z_i}}$$
where $\rho_i$ is the residual given by $\rho_i = (z-z_i)\left.\f {p(z)}{q(z)}\right|_{z=z_i}$ and $z_i$ are the roots of $q$. Therefore, by the linearity of the Toeplitz operator $T_n:L^2(\TT)\to \msf M_n$ we have the identity
$$\ts{T_n(p/q)=-\sum_{k=1}^{\deg q}\f {\rho_k}{z_k} Z_n(1/z_k)\, .}$$
Now the existence of $P_\vphi$ and $R$ follows from Proposition \ref{C+R}. Finally, if all the roots of $q$ are real, then the residuals $\rho_i$ are real, and
$$\ts{T_n(\Re (p/q))=-\sum_{k=1}^{\deg q}\f {\rho_k}{2z_k}(K_n(1/z_k)+I)}\, ,$$
which together with Proposition \ref{KMS} concludes the  proof.
\end{proof}
Notice that, if all the roots of $q$ are known, then one can explicitly compute the $\vphi$-circulant matrices of Theorem \ref{razio}, obtaining not exactly the optimal but a rank bounded preconditioner for an important class of Toeplitz matrices.

There follows one further lemma, whose proof can be found in \cite{lemma-zamarash} as underlined in \cite{zot}.
\begin{lemma}\label{approx}
 Let  $k \in \{1,\dots,n\}$ and $\alpha \in \RR$. For any $\ep >0$ there exist $a_i, b_i$
such that
$$\left | k^{-\alpha} - \sum_{i=1}^\rho a_i \(e^{- b_i}\)^k\right|\leq \ep k^{-\alpha}$$
with $\rho \leq \log \ep^{-1}(\beta_0 + \beta_1\log \ep^{-1} +\beta_2 \log n)$, and the coefficients $\beta_i$ depend only on $\alpha$. 
\end{lemma}
By using the result stated in Lemma \ref{C+R},  we can reformulate the above Lemma \ref{approx}. Let $\|\, \cdot \, \|_C$ denotes the Chebyshev norm  on $\msf M_n$, 
$$A = (a_{ij})_{ij} \in \msf M_n,\qquad \|A\|_C = \max_{ij}|a_{ij}|\, .$$
Consider the lower triangular Toeplitz matrix  $T_n=[(i-j)^{-\alpha}]_{i\geq j}$. Then, for any $\ep>0$, there exist $P_\vphi^{(k)} \in \C_\vphi$ and $R_k$ of of rank one, such that
$$\textstyle{\ln T_n - \sum_{k=1}^\rho a_k Z_n(e^{-b_k})\rn_C = \ln T_n -\sum_{k=1}^\rho P_\vphi^{(k)}+ R_k\rn_C \leq \ep\|T_n\|_C}$$
where $\rho$ is bounded as in Lemma \ref{approx}.  By taking the transpose of $T_n$ and then subtracting by $I$, one immediately observes (via Proposition \ref{KMS}) that also the symmetric Toeplitz matrix $T_n^s = (|i-j|^{-\alpha})_{ij}$ admits the decomposition 
$$\|T_n^s - C_\vphi -R\|_C \leq \ep \|T_n^s\|_C, \quad C_\vphi \in \C_\vphi$$
for a matrix $R$ whose rank is bounded by $2\rho$. 

From now on when referring to Lemma \ref{approx} we will always think at the latter two inequalities. 

Observe that the results obtained at this stage are enough to say that for any
polynomial $f$ and any \textit{symmetrized} polynomial $g(x)=f(|x|)$, the Toeplitz
matrices $T_n(f)$ and $T_n(g)$ admit the decomposition
$$T_n = C_\vphi + R + E$$
where $C_\vphi \in \C_\vphi$, $R$ has sufficiently small rank and $\|E\|_C \leq \ep$.
Notice furthermore that  the same can be said for continuous
symbol functions, since they can be approximated by polynomials (Weiestrass theorem).

An even better result can be obtained for the Toeplitz matrix whose
entries are positive integer powers of the indexes, namely
\begin{lemma}\label{p}
 Let $T_n = [(i-j)^p]_{i\geq j}$ and $T_n^s = (|i-j|^p)_{ij}$. Then for any $\vphi \in
\TT$, we have  $T_n=P_\vphi + R$ and $T_n^s = Q_\vphi + \tilde R$ with $P_\vphi, Q_\vphi \in
\C_\vphi$ and $\rank \tilde R \leq 2\rank R \leq 2(p+2)$. 
\end{lemma}
\begin{proof}
  If we prove the decomposition for $T_n$, the thesis for $T_n^s$ follows because of the
identity $T_n^s=T_n+T_n^\t$. Fix $p \in \NN$ and consider the polynomial $\chi_{\vphi}$
such that $\deg(\chi_\vphi)\leq p+1$ and  
$$\chi_{\vphi} (k) - \vphi \chi_\vphi(k-n)=k^p, \quad k=1.\dots,n-1\, .$$
Call $v_p$ the first column of $T_n$, i.e. $T_n = T_n(v_p, 0)$, then set
$$a(\chi_\vphi) =\matrix{c}{\chi_\vphi(0)\\ \chi_\vphi(1)\\ \vdots \\ \chi_\vphi(n-1)}, \quad b(\chi_\vphi) =\matrix{c}{\chi_\vphi(0)\\ \chi_\vphi(-1)\\ \vdots \\ \chi_\vphi(1-n)}\, .$$
It follows that $\vphi J b(\chi_\vphi) - \mPi_\vphi a(\chi_\vphi) = -\mPi_\vphi v_p$, and therefore 
$$\GE_\vphi(a(\chi_\vphi), b(\chi_\vphi)) = \GE_\vphi(v_p,0)\, ,\quad [T_n(a(\chi_\vphi), b(\chi_\vphi)), \mPi_\vphi] = [T_n(v_p,0), \mPi_\vphi]\, .$$
The thesis for $T_n$ follows by noting that $\rank T_n(a(\chi_\vphi), b(\chi_\vphi)) \leq p+2$.         
\end{proof}
Now consider a polynomial $f$ of degree $d$. By the previous Lemma we can affirm that the Toeplitz matrix whose entries are $f(|i-j|)$ can be decomposed into the sum of a $\vphi$-circulant matrix and a matrix $R$ whose rank is bounded by $\sum_{i=1}^d i+2=O(d^2)$. Such particular Toeplitz matrix is indeed a generalized KMS matrix. For the sake of completeness we recall that, given the matrix
$$K_n(f,\lambda) = \Bigl(f(|i-j|)\lambda^{|i-j|}\Bigr)_{ij}$$
a generalized KMS matrix is defined as
$$G_n = \sum_{k=1}^{m}\gamma_k K_n(f_k, \lambda_k)$$
where $\gamma_k$ and $\lambda_k$ are all real and $f_k$ are polynomials of degree $d_k$.

It is clear that using both Proposition \ref{KMS} and Lemma \ref{approx} we have  
$$\|G_n - P_\vphi + R\|_C\leq \ep \|G_n\|_C, \quad P_\vphi \in \C_\vphi\, ,$$
with $\rank R =O(\log \f 1 \ep(\log \f 1 \ep + \log n)\sum_{k=1}^m d_k )$.
\begin{proposition}\label{log}
 Set $f(z) = \log(z-z_0)$, $z_0, z \in \TT$. Then, for any $\ep>0$ there exist $P_\vphi \in \C_\vphi$ and $R_\ep$ with $\rank R_\ep \leq \log\ep^{-1}(\beta_0 + \beta_1\log \ep^{-1} +\beta_2 \log n)$ such that
$$\|T_n(f) - P_\vphi - R_\ep\|_C \leq \ep \|T_n(f)\|_C \, .$$
\end{proposition}
\begin{proof}
By the logarithmic singularity of $f$ there follows the equality
$$f(z) = \log z_0 + \sum_{k\geq 1}\f {z^k}{k z_0^k}\, ,$$
and thus for $i>j$, we have 
$$T_n(f)_{ij} = \f{1}{(i-j)z_0^{i-j}} = (i-j)^{-1} Z_n(z_{0}^{-1})_{ij} \, .$$
Note that the $\log z_0$ term gives rice to a multiple of the identity, thus to an element of $\C_\vphi$. Therefore we do not care about it. By Lemma \ref{C+R}, the matrix $Z_n(z_0)$ has the form $\tilde P_\vphi +R$ for a  $\tilde P_\vphi \in \C_\vphi$ and a rank one matrix $R$. Therefore by Lemma \ref{approx}, for any $\ep>0$ there exist $P_\vphi^{(k)}$ and a rank one $R_k$ such that 
$$\ln T_n(f)-(\tilde P_\vphi +R)\sum_k(P_\vphi^{(k)} +R_k)\rn_C \leq \ep \|T_n(f)\|_C\, .$$
As a consequence we have the thesis, since $(\tilde P_\vphi +R)\sum_k(P_\vphi^{(k)} +R_k)=P_\vphi + R_\ep$, with $R_\ep$ and $P_\vphi$ as in the statement.
\end{proof}
We can, finally, combine Theorem \ref{razio}, Lemma \ref{p}  and Proposition \ref{log}, to obtain 
\begin{theorem}\label{piece-wise-holo}
Let $f \in \L(\TT)$. For any $\ep >0$ there exist $P_\vphi,Q_\vphi \in \me C_\vphi$ such that
\begin{gather*}
 \|T_n(f)-P_\vphi-R_\ep\|_C\leq \ep \|T_n(f)\|_C=O(\ep)\, ,\\
\|T_n(\Re f)-Q_\vphi-\tilde R_\ep\|_C\leq \ep \|T_n(\Re f)\|_C=O(\ep)
\end{gather*}
with $\rank \tilde R_\ep \leq 2\rank R_\ep \leq 2\log\ep^{-1}(a + b\log\ep^{-1} + c \log n)+d$, and all the coefficients $a,b,c,d$ do not depend on $n$ neither on $\ep$.
\end{theorem}


The Hankel case can be discussed analogously. It is not difficult to check that $H_n(a,b) = JT_n(Ja,b)$, for any $a,b \in \CC^n$ with $a_1=b_1$. As before we call $H(T_n(a,b))$ such matrix. Therefore we can reformulate the results obtained in this section simply multiplying them by $J$ on the left, since this clearly does not affect the rank neither the arbitrariness of $\ep$. Moreover, we can write the Hankel matrix generated by the symbols
$$\zeta_\mu(\theta) = \f 1 {1-\mu e^{\i \theta}}, \qquad g(z) = \log(z-z_0), \quad z,z_0 \in \TT\, ,$$
in terms of $H(Z_n(\lambda))$. In fact, for instance, we have
\begin{equation}\label{hankel-mu}
 H_n(\zeta_\mu) =  \mu^{n-1}\, JZ_n(\mu^{-1})=\mu^{n-1} \, H(Z_n(\mu^{-1}))
\end{equation}
and a similar identity holds for $H_n(g)$.

We finally stress the fact that our initial problem is not well posed  when $f$ is a rational function and the linear system is defined by the Hankel matrix with symbol $f$. In fact such matrix $H_n(f)$ has in general a small rank which equals the number of poles of $f$ (due to the Kronecker theorem, 1881  \cite{kronecker}) and therefore the linear system $H_n(f)x=y$, when $n$ is large enough, could even be unsolvable.

\subsubsection*{Hartley-type algebras}
Let us conclude with few observations concerning Hartley-type algebras. In studying this case the
arbitrariness of $\vphi \in \TT$ is crucial, in fact it allows us to use both circulant
and
skew-circulant matrices and  thus to consider Hartley-type algebras. If $\H$ is a generic
Hartley-type
algebra, recall that  $\C_\vphi^s\subset \H$ with $\vphi \in \{-1,1\}$. 

Observe that by Proposition \ref{KMS} we already know that the KMS matrix $K_n(\lambda)$
admits the splitting $K_n(\lambda)=H+R$ where $\rank R=2$ and $H$ is an element of $\C_{\pm 1}^s \subset \H$.
In fact we have $K_n(\lambda) = C_\vphi(\xi_\vphi(\lambda)) + R$  where
 the matrix $C_\vphi(\xi_\vphi(\lambda))$, $\vphi\in \{1,-1\}$,  is circulant or
skew-circulant symmetric, respectively. Thus by the definitions in Table
\ref{tab:hartley}, when $\vphi \in \{1,-1\}$, $C_\vphi(\xi_\vphi(\lambda)) \in \H$.

Let us summarize this remark into the following
\begin{lemma}\label{H+R}
  Let $\lambda \in \RR$. For any  Hartley-type algebra $\H$ there exists $H \in \H$ and
$R$ of rank two, such that $K_n(\lambda) = H+ R$.
\end{lemma}
\begin{proof}
Specialize Proposition \ref{KMS} for $\vphi=1$ and $\vphi=-1$ and use the definitions in Table
\ref{tab:hartley}.
\end{proof}

Nonetheless we would stress the fact that all the results we obtained in terms of symmetric
$\vphi$-circulants may also be seen as involving
Hartley-type algebras. One just need to specialize them for $\vphi=1$ or $\vphi =-1$.
\section{Conclusions}
 We have tried to produce a first step towards the generalization of the ideas presented in
\cite{ot}. After a brief overview about matrix algebras of low complexity, their
generators and main properties, we have proposed a way  to extend the applicability of the
black-dot algorithm, proposed in \cite{ot} for the construction of optimal rank circulant
preconditioners for a Toeplitz system, to other types of linear systems (including Toeplitz plus Hankel like) and to preconditioners chosen in other low complexity matrix algebras. Then we have shown that, in fact, a suitable class of Toeplitz and
Hankel matrices is indeed representable as the sum of a $\vphi$-circulant matrix and a
small rank perturbation, for any $\vphi$ of modulus one. Combining such representation for $\vphi=1$ and
$\vphi=-1$ we then derive an analogous decomposition involving matrices from a
Hartley-type algebra and a low rank perturbation.

It is important to note that for a significant class of Toeplitz (and Hankel) matrices
associated with a rational symbol, the optimal rank $\vphi$-circulant and Hartley-type
preconditioners (as we called it) can be explicitly computed without the use of the
black-dot method, provided that the symbol function and its poles are explicitly known.

\subsubsection*{Acknowledgements}
This paper includes results from the master dissertation of the corresponding author. He thanks the researchers of INM for the support they gave him during his permanence in Moscow, and reserves a special and kind thank to his two supervisors.

\bibliographystyle{alpha}
\bibliography{bibliografiah}

\end{document}

%% file: myMath.tex
\DeclareFontFamily{OT1}{pzc}{}
\DeclareFontShape{OT1}{pzc}{m}{it}{<-> s * [1.050] pzcmi7t}{}
\DeclareMathAlphabet{\mathpzc}{OT1}{pzc}{m}{it}

\DeclareSymbolFont{EulerScript}{U}{eus}{m}{n} 
\SetSymbolFont{EulerScript}{bold}{U}{eus}{b}{n} 
\DeclareSymbolFontAlphabet\matheul{EulerScript}
 
\DeclareMathAlphabet{\mathant}{OMS}{antt}{m}{n} 
\DeclareMathAlphabet{\mathscrr}{OMS}{mdbch}{m}{n}

\newlanguage\fakelanguage
\newcommand\cyr{\fontencoding{OT2}\fontfamily{wncyr}\selectfont\language\fakelanguage}
\DeclareTextFontCommand{\textcyr}{\cyr}

\def\ts#1{\textstyle{#1}}

\def\THM{\begin{theorem}} \def\thm{\end{theorem}}
\def\COR{\begin{corollario}} \def\cor{\end{corollario}}
\def\PRO{\begin{proposition}} \def\pro{\end{proposition}} 
\def\OBS{\begin{oss}} \def\obs{\end{oss}}
\def\DFN{\begin{definition}} \def\dfn{\end{definition}}
\def\LEM{\begin{lemma}} \def\lem{\end{lemma}}
\def\PR{\begin{proof}} \def\pr{\end{proof}}
\newcommand{\EX}{\begin{esercizio}} \newcommand{\ex}{\end{esercizio}}
\newcommand{\ES}{\begin{esempio}} \newcommand{\es}{\end{esempio}}

\newcommand{\mb}{\mathbf}

\newcommand{\mc}{\mathcal}

\newcommand{\me}{\matheul}

\newcommand{\ma}{\mathant}
\newcommand{\mss}{\mathscrr}

\newcommand{\msf}{\mathsf}

\renewcommand{\i}{\mb{i}}

\renewcommand{\bar}{\overline}
\renewcommand{\tilde}{\widetilde}
\renewcommand{\hat}{\widehat}

\newcommand{\lista}{\begin{itemize}} \newcommand{\ei}{\end{itemize}}

\newcommand{\linea}{\rule[1mm]{\textwidth}{0.1mm} \\}

\def\({\left(}	\def\){\right)}
\def\[{\left[}	\def\]{\right]}
 
\def\<{\left<} \def\>{\right>}
\def\ln{\left\|} \def\rn{\right\|}

\def\um{\frac{1}{2}}
\newcommand{\ba}[1]{\begin{array}{#1}}  \newcommand{\ea}{\end{array}}

\def\Caption#1{\caption{\footnotesize{#1}}}

\def\ts#1{\textstyle{#1}}

\def\f#1#2{\frac{#1}{#2}}
\def\matrix#1#2{\left(\begin{array}{#1} #2 \end{array}\right)}

\def\3{\|\hskip-1pt |}

\newcommand{\Span}{\mathop{\text{Span}}}
\newcommand{\diag}{\mathop{\text{diag}}}

\DeclareMathOperator{\rank}{\text{rank}}

\DeclareMathOperator{\range}{\text{range}}

\DeclareMathOperator{\cas}{\text{cas}}
\DeclareMathOperator{\sd}{\mathrm {sd}}

\def\_#1{\scriptstyle{#1}}

\def\t{\text{\itshape{\textsf{T}}}}

\def\K{{\matheul K}}  \def\C{{\matheul C}}
  
\def\L{{\matheul L}}   
 \def\A{{\matheul A}} 
  
\def\H{{\matheul H}}  
  
  \def\T{{\matheul T}}

 \def\NN{{\mathbb{N}}}
\def\TT{{\mathbb T}} 
\def\ZZ{{\mathbb Z}} 
\def\CC{{\mathbb C}}  
\def\RR{{\mathbb R}}

   \def\vphi{{\varphi}}  \def\ep{{\varepsilon}}

\def\GE{\text{\textit{\textcyr{Zh}}}}

\def\mPi{\mathit{\Pi}}
\def\mOmega{\mathit{\Omega}}

\def\mDelta{{\mathit{\Delta}}}

\def\mTheta{{\mathit{\Theta}}}

\def\sse{\Leftrightarrow}
\def\iff{\Longleftrightarrow}

\def\to{\rightarrow}

\def\To{\longrightarrow}
